\documentclass[11pt, a4paper]{article}
\usepackage{amsfonts,amssymb,amsmath,amscd,latexsym,makeidx,theorem}

\newtheorem{thm}{Theorem}[section]
\newtheorem{lem}[thm]{Lemma}
\newtheorem{prop}[thm]{Proposition}

\newtheorem{defn}{Definition}[section]

\newenvironment{proof}{\noindent\emph{Proof.}}{\hfill$\square$\medskip}

\newcommand{\rn}{\mathbb{R}^{2m}}
\newcommand{\vol}{\mathrm{vol}}

\DeclareMathOperator{\diver}{div}
\DeclareMathOperator{\sign}{sign}
\DeclareMathOperator{\supp}{supp}

\title{Conformal metrics on $\rn$ with
constant Q-curvature,  prescribed volume and asymptotic behavior}

\date{January 5, 2014}

\author{Ali Hyder \thanks{The authors are supported by the Swiss National Science Foundation.} \\ {\small Universit\"at Basel}\\ {\small \texttt{ali.hyder@unibas.ch} }\and Luca Martinazzi${}^*$ \\ {\small Universit\"at Basel}\\ {\small \texttt{luca.martinazzi@unibas.ch} }}

\begin{document}

\maketitle

\begin{abstract}
We study the solutions $u\in C^\infty(\rn)$ of the problem
\begin{equation}\label{P0}
 (-\Delta)^mu=\bar Qe^{2mu}, \text{ where }\bar Q=\pm (2m-1)!, \quad V :=\int_{\rn}e^{2mu}dx <\infty,
\end{equation}
particularly when $m>1$. Problem \eqref{P0} corresponds to finding conformal metrics $g_u:=e^{2u}|dx|^2$ on $\rn$ with constant $Q$-curvature $\bar Q$ and finite volume $V$. Extending previous works of Chang-Chen, and Wei-Ye, we show that both the value $V$ and the asymptotic behavior of $u(x)$ as $|x|\to \infty$ can be simultaneously prescribed, under certain restrictions. When $\bar Q= (2m-1)!$ we need to assume $V<\vol(S^{2m})$, but surprisingly for $\bar Q=-(2m-1)!$ the volume $V$ can be chosen arbitrarily.
\end{abstract}

 \section{Introduction}
We consider the equation 
\begin{equation}\label{main-1}
 (-\Delta)^mu=(2m-1)!e^{2mu} ~~\text{in } \rn,
\end{equation}
where $u\in C^\infty(\rn)$ and satisfies
\begin{equation}\label{volume}
V :=\int_{\rn}e^{2mu} dx<\infty.
\end{equation}
Equation \eqref{main-1} has been widely studied because of its geometric meaning. Indeed if $u$ solves \eqref{main-1}, then the conformal metric $g_u:=e^{2u}|dx|^2$ on $\rn$ (here $|dx|^2$ denotes the Euclidean metric on $\rn$) has constant $Q$-curvature equal to $(2m-1)!$. For a brief discussion of the geometric meaning of \eqref{main-1} and a survey of related previous works we refer to the introduction of \cite{LM-classification} and the references therein.
Here we only mention some relevant facts, necessary to contextualize the results of our present work.

First of all the assumption that $u\in C^\infty(\rn)$ is not restrictive, since any weak solution $u\in L^1_{\mathrm{loc}}(\rn)$ of \eqref{main-1} with right-hand side in $L^1_{\mathrm{loc}}(\rn)$ is smooth, see e.g. \cite[Corollary 8]{LM-classification}. Also the particular choice of the constant $(2m-1)!$ in \eqref{main-1} is not restrictive, since it can be changed by considering $u+C$ for $C\in \mathbb{R}$.

Next we recall that  Problem \eqref{main-1}-\eqref{volume} possesses the following explicit radially symmetric solutions
$$u(x)=\log\left(\frac{2\lambda}{1+\lambda^2|x-x_0|^2}\right),\quad \lambda>0, \,x_0\in\rn,$$
which are called spherical solutions, since they are obtained (up to a M\"obius transformation) by pulling back the round metric of $S^{2m}$ onto $\rn$ via the stereographic projection.

While in dimension $2$, i.e. for $m=1$, such spherical solutions exhaust the set of solutions to \eqref{main-1}-\eqref{volume}, as proven by W. Chen and C. Li \cite{CL}, in the case $m\ge 2$ A. Chang and W. Chen \cite{CC} showed that non-spherical solutions do exist. In fact they proved that for any $m\ge 2$ and every $V\in (0,\vol(S^{2m}))$ there exists a (non-spherical) solution to \eqref{main-1}-\eqref{volume}. This suggests to investigate the properties of such solutions. 
Building upon the previous work of A. Chang and P. Yang \cite{CC}, C-S. Lin for $m=2$ and L. Martinazzi for $m>2$ proved:

\medskip

\noindent\textbf{Theorem A (\cite{Lin}, \cite{LM-classification})} \emph{If $u$ solves \eqref{main-1}-\eqref{volume}, then $u$ has the asymptotic behavior
\begin{equation}\label{uasym}
u(x)=-\alpha \log(|x|) -P(x)+C+o(1),\quad o(1)\to 0\text{ as }|x|\to\infty,
\end{equation}
where $\alpha=\frac{2V}{\vol(S^{2m})}$ and $P$ is a polynomial of degree at most $2m-2$ bounded from below. Moreover $P$ is constant if and only if $u$ is spherical. When $m=2$ one has $V\in (0,\vol(S^4)]$ and $V=\vol(S^4)$ if and only if $u$ is spherical.}

\medskip

J. Wei and D. Ye complemented the result of C-S. Lin by showing, among other things:

\medskip

\noindent\textbf{Theorem B (\cite{W-Y})} \emph{For any $V\in (0,\vol(S^4))$ and $P(x)=\sum_{j=1}^4 a_jx_j^2$ with $a_j>0$, Problem (\ref{main-1})-\eqref{volume} has a solution with asymptotic expansion \eqref{uasym} for some $C\in\mathbb{R}$.}

\medskip

The first result which we prove here is an extension of the result of J. Wei and D. Ye to the case $m>2$.
We will prove the existence of solutions to \eqref{main-1}-\eqref{volume} having the asymptotic behavior \eqref{uasym}
where $P$ will be any given polynomial of degree at most $2m-2$ satisfying
\begin{equation}\label{poly}
\lim_{|x|\to \infty} x\cdot \nabla P(x)=\infty,
\end{equation}
while $\alpha>0$ is determined by $V\in (0,\vol(S^{2m}))$.
More precisely, define
$$\mathcal{P}_m:=\left\{P \text{ polynomial in }\rn : \deg P\le 2m-2, \,\eqref{poly} \text{ holds} \right\}.$$
It is worth noticing that \eqref{poly} is equivalent to the apparently stronger condition
\begin{equation}\label{poly2}
\liminf_{|x|\to \infty}\frac{P(x)}{|x|^a}>0\quad \text{and}\quad \liminf_{|x|\to \infty}\frac{x\cdot \nabla P(x)}{|x|^a}>0,\quad \text{for some }a>0.
\end{equation}
Indeed \eqref{poly} implies the second inequality of \eqref{poly2} by a subtle result of E. Gorin (see \cite[Theorem 3.1]{gor}), and the second inequality in \eqref{poly2} implies the first one, since one can write 
$$P(x)=\int_0^{|x|}\frac{d}{dr}P\left(r\frac{x}{|x|}\right)dr +P(0).$$
A simple example of polynomial belonging to $\mathcal{P}_m$ is
$$ P(x)=\sum_{j=1}^{2m}a_jx_j^{2i_j} +p(x),$$
where $a_j>0$, $i_j\in \{1,2,\dots,m-1\}$ for $1\le j\le 2m$, and $p$ is a polynomial of degree at most $2\min\{i_j\}-1$, but in general $\mathcal{P}_m$ contains polynomials whose higher degree monomials do not split in such a simple way.

\begin{thm}\label{thm_2}
For any integer $m\ge 2$, given $P\in \mathcal{P}_m$  and $V\in (0,\vol(S^{2m}))$,  
 there exists a solution of \eqref{main-1}-\eqref{volume} having the asymptotic behavior \eqref{uasym} with $\alpha=\frac{2V}{\vol(S^{2m})}$.
 \end{thm}

The restriction $V<\vol(S^{2m})$ in Theorem \ref{thm_2} is necessary when $m=2$ because of the result of C-S. Lin (Theorem A), but appears to be only a technical issue when $m\ge 3$.
In fact for $m=3$ L. Martinazzi recently proved that there are solutions to (\ref{main-1})-\eqref{volume} with $V$ arbitrarily large, see \cite{LM-volume}.
 The crucial step in which we need $V$ to be smaller than $\vol(S^{2m})$ is Theorem \ref{lemma_1} below, a compactness result which follows form the blow-up analysis of sequences of prescribed $Q$-curvature in open domains of $\rn$ (Theorem \ref{useful} below) proven by L. Martinazzi, and inspired by previous works of H. Br\'ezis and F. Merle \cite{BM} and F. Robert \cite{rob}. This compactness is used to prove the a priori bounds necessary to run the fixed point argument of \cite{W-Y}, which we closely follow. For $m>2$ it remains open whether one can prescribe $P\in \mathcal{P}_m$ and $V\ge \vol(S^{2m})$ in Theorem \ref{thm_2}.

From the work of Br\'ezis-Merle we also borrow a simple but fundamental critical estimate, whose generalization is Lemma \ref{a2m} below, which is used in Lemma \ref{complement} below. 

\medskip

As we shall now show, things go differently when the prescribed $Q$-curvature is negative.
Consider the equation
\begin{equation}\label{main-2}
 (-\Delta)^mu=-(2m-1)!e^{2mu} ~~\text{in }\rn,
\end{equation}
whose solutions give rise to metrics $g_u=e^{2u}|dx|^2$ of $Q$-curvature $-(2m-1)!$ in $\rn$. One can easily verify that under the assumption \eqref{volume} Equation \eqref{main-2} has no solutions when $m=1$, see e.g. \cite[Proposition 6]{LM-negative}. On the other hand, when $m\ge 2$ we have:

\medskip

\noindent\textbf{Theorem C (\cite{LM-negative})} \emph{For every $m\ge 2$ there is some $V>0$ such that Problem \eqref{main-2}-\eqref{volume} has a radially symmetric solution. Every solution to \eqref{main-2}-\eqref{volume} (a priori not necessarily radially symmetric) has the asymptotic behavior given by \eqref{uasym}
where $\alpha=-\frac{2V}{\vol(S^{2m})}$ and $P$ is a \emph{non-constant} polynomial of degree at most $2m-2$ bounded from below.}

\medskip

Notice that, contrary to Chang-Chen's result \cite{CC}, the existence part of Theorem C does not allow to prescribe $V$. Moreover its proof is based on an ODE argument which only produces radially symmetric solutions. It is then natural to address the following question: For which values of $V$ and which polynomials $P$ does Problem \eqref{main-2}-\eqref{volume} have a solution with asymptotic behavior \eqref{uasym} (with $\alpha=-\frac{2V}{\vol(S^{2m})}$)? 
In analogy with Theorem \ref{thm_2} we will show:

\begin{thm}\label{thm_1}
For any integer $m\ge 2$, given $P\in \mathcal{P}_m$  and $V>0$, 
 there exists a solution of \eqref{main-2}-\eqref{volume} having the asymptotic behavior \eqref{uasym} for $\alpha=-\frac{2V}{\vol(S^{2m})}$.
\end{thm}

The remarkable fact which allows for large values of $V$ in Theorem \ref{thm_1} (but not in Theorem \ref{thm_2}) is that, as shown in \cite{LM-quantization}, when the $Q$-curvature is negative, compactness is obtained even for large volumes, compare Theorems \ref{useful} and \ref{lemma_1} below. This in turn depends on Theorem C above, and in particular on the fact that the polynomial in the expansion \eqref{uasym} of a solution to \eqref{main-2}-\eqref{volume} is necessarily non-constant. 

\medskip

About the assumption that $P\in \mathcal{P}_m$ in Theorems \ref{thm_2} and \ref{thm_1}, we do not claim nor believe that it is optimal, but it is technically convenient in the crucial Lemma \ref{int<epsilon} below, where it is needed in \eqref{stimapo}. Since a solution to \eqref{main-1}-\eqref{volume} or \eqref{main-2}-\eqref{volume} must satisfy $\eqref{uasym}$ for $\alpha=\pm \frac{2V}{\vol(S^{2m})}$, a necessary condition on $P$ and $V$ is
\begin{equation}\label{assp}
\int_{\rn\setminus B_1} e^{-2m(P(x)+\alpha\log|x|)}dx<\infty,
\end{equation}
but it is unknown whether this condition is also sufficient to guarantee the existence of a solution to \eqref{main-1}-\eqref{volume} or \eqref{main-2}-\eqref{volume} with asymptotic expansion \eqref{uasym}, at least in the negative case, or for $V<\vol(S^{2m})$ in the positive case.

Also replacing \eqref{poly} with the weaker assumption
\begin{equation}\label{poly3}
\lim_{|x|\to \infty} P(x)=\infty
\end{equation}
 (which implies the first inequality in \eqref{poly2}, hence \eqref{assp}) creates problems, since \eqref{poly3} does not imply \eqref{poly} when $\deg P \ge 4$, see e.g. Proposition \ref{proppoly} in the appendix, and as already noticed \eqref{poly} is crucial in Lemma \ref{int<epsilon} below.

\medskip

Finally, we remark that new difficulties arise when recasting the above problems in odd dimension. For instance in dimension $3$ T. Jin, A. Maalaoui, J. Xiong and the second author studied in  \cite{JMMX} the non-local problem
$$(-\Delta)^\frac32 u= 2e^{3u}\quad \text{in }\mathbb{R}^3,\quad V:=\int_{\mathbb{R}^3}e^{3u}dx<\infty,$$
proving the existence of some non-spherical solutions with asymptotic behavior as in \eqref{uasym}. Whether also in this case one can show an analog to Theorems \ref{thm_2} and \ref{thm_1} above is an open question.

\medskip


 \noindent\textbf{Notation} In the following $C$ will denote a generic positive constant, whose dependence will be specified when necessary, and whose value can change from line to line. We will also write
$$B_r(x):=\{y\in \rn:|y-x|<r\},\quad B_r:=B_r(0).$$

\section{Strategy of the proof of Theorems \ref{thm_2} and \ref{thm_1}}

Fix $u_0\in C^{\infty}(\rn)$ such that $u_0(x)=\log |x|$ for $|x|\geq 1$. Integration by parts yields
\[
\int_{\rn}(-\Delta)^mu_0dx=-\gamma_m,
\]
where $\gamma_m$ is defined by
\begin{equation}\label{gammam}
(-\Delta)^m\log\frac{1}{|x|}=\gamma_m\delta_0 \text{ in }\rn, \text{ i.e. }\gamma_m=\frac{(2m-1)!}{2}\vol(S^{2m}).
\end{equation}
Let $V$, $\alpha=\pm \frac{2V}{\mathrm{vol}(S^{2m})}$ and $P\in \mathcal{P}_m$ be given as in Theorem \ref{thm_2} or \ref{thm_1}. We would like to find a solution to \eqref{main-1} or \eqref{main-2} of the form
\begin{equation}\label{uexp}
u= -\alpha u_0-P+v+C,
\end{equation}
for a suitable choice of $C\in \mathbb{R}$ and of a smooth function $v(x)=o(1)$  as $|x|\to\infty.$
Define
\begin{equation}\label{defK}
K=\frac{\alpha\gamma_m}{V}e^{-2mP- 2m\alpha u_0}=\sign(\alpha) (2m-1)!e^{-2mP- 2m\alpha u_0},
\end{equation}
and notice that \eqref{poly} implies
\begin{equation}\label{decayK}
|K(x)|\le C_1 e^{-C_2|x|^a}
\end{equation}
for some $C_1, C_2>0$.

Now if we assume \eqref{volume}, then the constant $C$ in \eqref{uexp} is determined by the function $v$. Indeed \eqref{volume} implies
$$V=\int_{\rn}e^{2mu}dx=\frac{e^{2mC}}{(2m-1)!}\int_{\rn} |K| e^{2mv}dx,$$
hence
\begin{equation}\label{defcv}
C=c_v:=-\frac{1}{2m}\log\left(\frac{1}{(2m-1)! V}\int_{\rn}|K|e^{2mv} dx\right)=-\frac{1}{2m}\log\left(\frac{1}{\alpha\gamma_m}\int_{\rn}Ke^{2mv} dx\right).
\end{equation}
An easy computation shows that $u$ given by \eqref{uexp} satisfies
\[
 (-\Delta)^mu=\mathrm{sign}(\alpha)(2m-1)!e^{2mu}
\]
and \eqref{volume} if and only if $C=c_v$ and
\begin{equation}\label{defv}
 (-\Delta)^mv = Ke^{2m(v+c_v)}+\alpha(-\Delta)^mu_0.
\end{equation}
Then we will use a fixed point method in the spirit of \cite{W-Y} to find a solution $v$ to \eqref{defv} in the Banach space
$$C_0(\rn):=\left\{f\in C^0(\rn):\lim_{|x|\to\infty} f(x)=0\right\},\quad \|f\|_{C_0}:=\sup_{\rn}|f|,$$
and of course $v$ will also be smooth by elliptic estimates. In order to run the fixed-point argument we introduce the following weighted Sobolev spaces.

\begin{defn} For $k\in \mathbb{N}$, $\delta\in\mathbb{R}$ and $p\ge 1$ we set $M_{k,\delta}^p(\rn)$ to be the completion of $C_c^{\infty}(\rn)$ in the norm 
\[
\|f\|_{M_{k,\delta}^p} := \sum_{|\beta|\leq k}\|(1+|x|^2)^{\frac{(\delta+|\beta|)}{2}}D^{\beta}f\|_{L^p(\rn)}.
\]
We also set $L^p_{\delta}(\rn):= M_{0,\delta}^p(\rn)$. Finally we set 
$$\Gamma^p_\delta(\rn)  := \left\{f\in L^p_{2m+\delta}(\rn):\int_{\rn}f dx=0\right\},$$
whenever $\delta p>-2m$, so that $L^p_{2m+\delta}(\rn)\subset L^1(\rn)$ and the above integral is well defined.
\end{defn}

\begin{lem}\label{lemmacv}
Fix  $p\ge 1$ and $\delta>-\frac{2m}{p}$.  For $v\in C_0(\rn)$ and $c_v$  as in \eqref{defcv} we have 
$$S(v):=Ke^{2m(v+c_v)}+\alpha(-\Delta)^m u_0\in \Gamma^p_\delta(\rn),$$
and the map $S:C_0(\rn)\to \Gamma^p_\delta(\rn)$ is continuous.
\end{lem}

\begin{proof} This follows easily from \eqref{decayK} and dominated convergence.
\end{proof}

\begin{lem}[Theorem 5 in \cite{McOwen}]\label{isomorphism} 
For $1<p<\infty$ and $\delta\in \left(-\frac{2m}{p},-\frac{2m}{p}+1\right)$, the operator $(-\Delta)^m$ is an isomorphism from $M_{2m,\delta}^p(\rn)$ to $\Gamma^p_\delta(\rn)$. 
\end{lem}

The following Lemma will be proven in Section \ref{proofcompact} below.

\begin{lem}\label{compact}
 For $\delta>-\frac{2m}{p}$, $p\ge 1$, the embedding  $E:M_{2m,\delta}^p(\rn)\hookrightarrow C_0(\rn)$ is compact.
\end{lem}


Fix $p\in (1,\infty)$ and $\delta\in \left( -\frac{2m}{p}, -\frac{2m}{p}+1  \right)$. 
Then by Lemma \ref{lemmacv}, Lemma \ref{isomorphism} and Lemma \ref{compact}, one can define a compact map
\begin{equation}\label{defT}
T:=E\circ ((-\Delta)^m)^{-1}\circ S: C_0(\rn)\to C_0(\rn)
\end{equation}
given by
$Tv  =\bar{v}$ where $\bar v$ is the only solution to 
$$  (-\Delta)^m\bar{v} =Ke^{2m(v+c_v)} + \alpha (-\Delta)^mu_0,$$
and compactness follows from the continuity of $S$ and $((-\Delta)^m)^{-1}$ and the compactness of $E$.

If $v$ is a fixed point of $T$, then it solves \eqref{defv} and $u=v+c_v-P-\alpha u_0$ is a solution of \eqref{main-1} or
 \eqref{main-2} (depending on the sign of $K$ in \eqref{defK}) and \eqref{volume}, with asymptotic expansion \eqref{uasym}. Then in order to prove 
Theorems \ref{thm_2} and \ref{thm_1} it remains to prove that $T$ has a fixed point, 
and we shall do that using the following fixed-point theorem.

\begin{lem}[Theorem 11.3 in \cite{G-T}]\label{leray}
 Let $T$ be a compact mapping of a Banach space $X$ into itself, and suppose that there exists a constant $M$ such that
 $$\|x\|_{X}<M$$
 for all $x\in X$ and $t\in (0,1]$ satisfying $tTx=x$. Then $T$ has a fixed point.
\end{lem}

In order to apply Lemma \ref{leray} to the operator $T$ defined in \eqref{defT} we will  prove in Section \ref{bounds} the following a priori bound, which completes the proof of Theorems \ref{thm_2} and \ref{thm_1}.

\begin{prop}\label{for-deg-cond} For any $0< t\leq 1$ and $v\in C_0(\rn)$ such $tTv=v$ we have
\begin{align}
 \|v\|_{C_0(\rn)}\leq M,
\end{align}
with $M$ independent of $v$ and $t$.
\end{prop}

\section{A priori estimates and proof of  Proposition \ref{for-deg-cond}}\label{bounds}
 Throughout this section let $t\in (0,1]$ and $v\in C_0(\rn)$ be fixed and satisfy $tTv=v$, that is
\[
 (-\Delta)^mv=t(Ke^{2m(v+c_v)}+\alpha(-\Delta)^mu_0),
\]
where $c_v$ is as in \eqref{defcv}.
Also define 
\begin{align}\label{expression_barw}
 \bar{w}:=v+c_v+\frac{\log t}{2m}.
\end{align}


\begin{lem}\label{lemmav} We have
\begin{equation}\label{expression_v}
 v(x)=-\frac{t}{\gamma_m}\int_{\rn}\log(|x-y|)K(y)e^{2m(v(y)+c_v)}dy+t\alpha u_0(x).
\end{equation}
\end{lem}

\begin{proof}
Let $\tilde{v}(x)$ be defined as the right-hand side of \eqref{expression_v}.
Then  for $|x|\geq 1$, using \eqref{defcv} we write 
\[
 \tilde{v}(x)=\frac{t}{\gamma_m}\int_{\rn}K(y)e^{2m(v(y)+c_v)}(\log|x|-\log|x-y|)dy
\]
We first show that
\begin{equation}\label{limv}
\lim_{|x|\to\infty}\tilde{v}(x)=0.
\end{equation}
Let $R>1$ be fixed. Then for $|x|>2R$, we split
$$\tilde{v}(x)=\sum_{i=1}^5 I_i,\quad I_i:= \frac{t}{\gamma_m}\int_{A_i}K(y)e^{2m(v(y)+c_v)}\log \left(\frac{|x|}{|x-y|}\right)dy,$$
where
\begin{equation*}
\begin{split}
A_1&:=B_R(0)\\
A_2&:=B_1(x)\\
A_3&:=B_{|x|/2}(x)\setminus B_1(x)\\
A_4&:=(B_{2|x|}(x)\setminus B_{|x|/2}(x))\setminus B_R(0)\\
A_5&:= \rn \setminus B_{2|x|}(x),
\end{split}
\end{equation*}
and we will show that $I_i\to 0$ as $|x|\to\infty$ for $1\le i\le 5$.

For $i=1$, since $\lim_{|x|\to\infty}\log \left(\frac{|x|}{|x-y|}\right)= 0$  uniformly with respect to $y\in B_R(0)$, from the
dominated convergence theorem we get
\[
 |I_1| \leq C\int_{B_R(0)}|K(y)|\left|\log\left(\frac{|x|}{|x-y|}\right)\right|dy\to 0 \quad \text{as }|x|\to\infty.
\]
From \eqref{decayK} we also have
\begin{align}
 |I_2| &\leq C\int_{B_1(x)}|K(y)|\left(\log|x|+|\log|x-y||\right)dy \notag\\
  & \leq C\|K\|_{L^{\infty}(B_1(x))}\left(\log|x|+\|\log|\cdot|\|_{L^1(B_1(0))} \right)\notag\\
 &\rightarrow 0, \quad \text{as } |x|\to\infty. \notag
\end{align}
Since \eqref{decayK} yields $K\log(|\cdot|)\in L^1(\rn)$, we infer with the dominated convergence theorem
\begin{align}
 |I_3|&\leq C\int_{\{1\le|x-y|<|x|/2\}}|K(y)|\left(\log|x|+\log(|x|/2)\right)dy \notag\\
 &\leq C\int_{\{1\le |x-y|<|x|/2\}}|K(y)|\left(\log|2y|+\log(|y|)\right)dy\notag\\
 & \rightarrow 0,\quad  \text{as }|x|\to\infty.\notag
\end{align}
Using that $\frac{1}{2}<\frac{|x|}{|x-y|}<2$ on $A_4$ and that $K\in L^1(\rn)$ we find that for every
$\varepsilon>0$ it is possible to choose $R$ so large that
$$|I_4|\leq C\int_{A_4}|K(y)|\left|\log\left(\frac{|x|}{|x-y|}\right)\right|dy \leq C\int_{A_4}|K|dy\le C \int_{\rn\setminus B_R(0)} 
|K|dy\le\varepsilon.$$
Finally, again using that $K\log(|\cdot|)\in L^1(\rn)$ with the dominated convergence theorem we get
\begin{align}
 |I_5|&\leq C\int_{\{|x-y|>2|x|\}}|K(y)|(\log|x|+\log|x-y|)dy \notag\\
 &\leq C\int_{\{|x-y|>2|x|\}}|K(y)|(\log|y|+\log|2y|)dy  \notag\\
  & \rightarrow 0, ~~\text{as } |x|\to\infty.\notag
\end{align}
Since $\varepsilon$ can be chosen arbitrarily small,  \eqref{limv} is proven. Since $v\in C_0(\rn)$, and $\Delta^m \tilde v=\Delta^m v$, 
the difference $w:=v-\tilde v$ satisfies
\[
\Delta^mw=0 \quad \text{in }\rn, \qquad \lim_{|x|\to\infty} w(x)=0.
\]
Then by the Liouville theorem for polyharmonic functions (see e.g. Theorem 5 in 
\cite{LM-classification}) $w$ is a polynomial, and since it vanishes at infinity, it must be identically zero, i.e. $v\equiv\tilde{v}$. 
\end{proof}

By Lemma \ref{isomorphism} and \eqref{decayK}, we have
 
\begin{align}
\frac{1}{C} \|v\|_{M_{2m,\delta}^p}&\le  \|(-\Delta)^mv\|_{L^p_{2m+\delta}} \notag\\
&= \|Ke^{2m\bar{w}}+t\alpha(-\Delta)^mu_0\|_{L^p_{2m+\delta}} \label{estimatev} \notag\\
                       & \leq \|K\|_{L^p_{2m+\delta}} \|e^{2m\bar{w}}\|_{L^\infty}   + \alpha\|(-\Delta)^mu_0\|_{L^p_{2m+\delta}} \notag\\
& \le C  \|e^{2m\bar{w}}\|_{L^\infty} + C,\notag
\end{align}
with $C$ independent of $t$ and $v$,
and together with Lemma \ref{local} and Lemma \ref{complement} below  we obtain
$$\|v\|_{M_{2m,\delta}^p}\le C,$$
where $C$ is independent of $v$ and $t$. Now Proposition \ref{for-deg-cond} follows at once from the continuity of the embedding
$M_{2m,\delta}^p(\rn)\hookrightarrow C_0(\rn)$ (see Lemma \ref{compact}).\\

\noindent\textbf{Remark.}  An alternative way of getting uniform bounds on $\|v\|_{C_0}$ is to get uniform upper bounds of $\bar{w}$ and use them in \eqref{expression_v}.

\medskip

Using Lemma \ref{lemmav} one can prove the following decay estimate for the derivatives of $v$ at infinity.

\begin{lem}\label{estimates_infty} 
For $1\le \ell\le 2m-1$ we have
$$   \lim_{|x|\to\infty} |x|^\ell\nabla^\ell v(x)=\lim_{|x|\to\infty} |x|^\ell\nabla^\ell \bar w(x)=0.$$
\end{lem}

\begin{proof} Notice that $\nabla v=\nabla\bar w$, so it is enough to work with $v$.

Using \eqref{expression_v} for $|x|>1$ one can compute 
$$ \nabla^\ell v(x)=\frac{1}{\gamma_m}\int_{\rn}K(y)e^{2m\bar{w}(y)}\left(\nabla^\ell \log(|x|)-\nabla^\ell\log(|x-y|)\right)dy.$$
Fix $\varepsilon>0$ and $R_1>1$ such that
 $$  \int_{\rn\setminus B_{R_1}}|K|e^{2m\bar{w}} dy<\varepsilon.$$
 For $|x|>2R_1$, we split $\rn$ in to three disjoint domains:
$$A_1:=B_{R_1}(0),\quad A_2:=B_{|x|/2}(x),\quad A_3:=\rn\setminus(A_1\cup A_2).$$
Then  
  $$ |x|^\ell \nabla^\ell v(x)=\frac{1}{\gamma_m}\sum_{i=1}^3 I_i, 
  \quad I_i:=|x|^\ell\int_{A_i}K(y)e^{2m\bar{w}(y)}\left(\nabla^\ell \log(|x|)-\nabla^\ell\log(|x-y|)\right)dy. $$ 
Since $R_1$ is fixed, for $|x|$ large enough we have by the mean-value theorem 
$$\left|\nabla^\ell \log(|x|)-\nabla^\ell\log(|x-y|)\right|\le|y| \sup_{B_{|y|}(x)}\left|\nabla^{\ell+1}\log(|z|)\right|\le \frac{C}{|x|^{\ell+1}} \quad \text{for } y\in A_1,$$
hence with \eqref{defcv} we get
 \begin{align}
  |I_1|\leq \frac{C}{|x|}\int_{A_1}|K|e^{2m\bar{w}}dy \leq \frac{C}{|x|}|\alpha|\gamma_m\rightarrow 0, \quad \text{as } |x|\to\infty.
  \notag
 \end{align}
Since $K$ goes to zero rapidly at infinity, $\bar w$ is bounded, and  $|x-y|\le |x|/2$ on $A_2$, we have
 \begin{align}
  |I_2|&\leq C\|K\|_{L^{\infty}(A_2)} \|e^{2m\bar{w}}\|_{L^{\infty}}|x|^\ell\int_{A_2}\left(\frac{1}{|x|^\ell}+\frac{1}{|x-y|^\ell}\right)dy \notag\\
   &\leq C \|K\|_{L^{\infty}(A_2)}\|e^{2m\bar{w}}\|_{L^{\infty}}|x|^{2m}\notag\\
   & \to 0, \qquad \text{as } |x|\to\infty. \notag
 \end{align}
 On $A_3$ we have $|x-y|\geq |x|/2$, which implies $\frac{|x|^\ell}{|x-y|^\ell}\leq 2^\ell$. Hence  
 $$ |I_3|\leq C(1+2^\ell)\int_{A_3}|K|e^{2m\bar{w}} dy< C\varepsilon.$$
Since $\varepsilon$ is arbitrarily small, the proof is complete.
\end{proof}

\begin{lem}\label{local}
The function $\bar{w}$ given by \eqref{expression_barw} is locally uniformly upper bounded, i.e. for every $R>0$ there exists $C=C(R)$ such that $\bar w\le C$ in $B_R$.
\end{lem}
\begin{proof}
 Since $u_0$ is a fixed function and locally bounded, it is enough to prove that $w:=\bar{w}-t\alpha u_0$
 is locally uniformly upper bounded. 
 Now 
 \begin{align}
 (-\Delta)^mw  =tKe^{2m(v+c_v)}= Qe^{2mw}, \notag
\end{align}
where $Q=Ke^{2mt\alpha u_0}$.\\
We bound
\begin{align}
 \int_{B_R}e^{2mw}dx  =t\int_{B_R}e^{2m(v+c_v)-2mt\alpha u_0} dx \leq C(R)\int_{B_R}|K|e^{2m(v+c_v)}dx
                    \leq C(R)|\alpha|\gamma_m, \notag
\end{align}
where we used \eqref{defcv}
and that $|K|$ is positive and continuous.

In addition in the case when $Q>0$ we have
\begin{equation*}
\int_{B_R}Qe^{2mw}dx\leq \int_{B_R}Ke^{2m(v+c_v)}dx < \alpha\gamma_m<(2m-1)!|S^{2m}|.
\end{equation*}
Moreover Lemma \ref{lemmav} gives 
\[
 \Delta w(x)=-\frac{t}{\gamma_m}\int_{\rn}\frac{2m-2}{|x-y|^2}K(y)e^{2m(v(y)+c_v)}dy
\]
and with Fubini's theorem we get
\begin{align}
 \int_{B_R}|\Delta w(x)|dx & =\frac{t}{\gamma_m}(2m-2)\int_{\rn}|K(y)|e^{2m(v(y)+c_v)}
                   \left(\int_{B_R}\frac{dx}{|x-y|^2}\right)dy \notag\\
   & \leq C\int_{\rn}|K(y)|e^{2m(v(y)+c_v)}\left(\int_{B_R(y)}
                          \frac{dx}{|x-y|^2}\right)dy \notag\\
     &\leq CR^{2m-2}. \notag
\end{align}
Therefore  Theorem \ref{lemma_1} implies that there exists $C=C(R)>0$
(independent of $w$) such that
\[
 \sup_{B_{R/2}}w\leq C.
\]
\end{proof}

A consequence of the local uniform upper bounds of $\bar{w}$ is the following local uniform bound for the derivatives of $v$:
 
\begin{lem}\label{local-derivatives}
 For every $R>0$ there exists a conastant $C=C(R)>0$ independent of $v$ and $t$ such that for $1\le \ell\le 2m-1$ we have
 \begin{equation*}
  \sup_{B_R}|\nabla^\ell v| \leq C.
 \end{equation*}

 \end{lem}
\begin{proof}
 Let $x\in B_R$. Then from \eqref{expression_v} and Lemma \ref{local}, we have
 \begin{align}
  |\nabla^\ell (v-t\alpha u_0)|& \le C\int_{\rn}|K(y)|e^{2m\bar{w}(y)}\frac{1}{|x-y|^{\ell}}dy\notag\\
       & \le C\|K\|_{L^{\infty}}\|e^{2m\bar{w}}\|_{L^{\infty}(B_{2R})}\int_{B_{2R}}\frac{1}{|x-y|^{\ell}}dy +
              \frac{C}{R^{\ell}}\int_{\rn\setminus B_{2R}}|K|e^{2m\bar{w}}dy   \notag\\
        & \leq C(R),\notag      
 \end{align}
where the last integral is bounded using \eqref{defcv}.
Since $u_0$ is smooth, $\alpha$ is fixed and $t\in (0,1]$, then the lemma follows.
\end{proof}

Now to prove uniform upper bounds for $\bar{w}$ outside a fixed compact set, first we will need the following result, which relies on a Pohozaev-type identity.
\begin{lem}\label{int<epsilon}
 For given $\varepsilon>0$, there exists $R_0=R_0(\varepsilon)>0$ only depending on $K$ (and not on $v$ or $t$) such that 
\[
   \int_{\rn\setminus B_{R_0}}|K|e^{2m\bar{w}}dx<\varepsilon.
\]
\end{lem}

 \begin{proof}
 Taking $R\to\infty$ in Lemma \ref{pohozaev} and noticing that the first term on the right-hand side of \eqref{eqpoho1} vanishes thanks to \eqref{decayK} and last two terms  vanish thanks to Lemma \ref{estimates_infty}, we find
 \begin{equation}\label{eqxnabla}
   \int_{\rn}(x\cdotp\nabla K)e^{2m\bar{w}}dx + 2m\int_{\rn}Ke^{2m\bar{w}}dx 
   -2mt\alpha\int_{B_1}(x\cdotp\nabla  v)(-\Delta)^mu_0dx   = 0.
 \end{equation}
Thanks to \eqref{poly2} we can find $C_1>0$ and $R_1\ge 1$ such that 
\begin{equation}\label{stimapo}
 x\cdotp\nabla |K(x)|=-2m\left(x\cdot \nabla P(x)+\alpha\right)|K(x)| \leq -\frac{1}{C_1}|x|^a|K(x)|\quad \text{for }|x|\ge R_1.
\end{equation}
Then for some $R\ge R_1$ to be fixed later we bound
\begin{equation}\label{last}
\begin{split}
\frac{1}{C_1} R^a\int_{\rn\setminus B_R}|K|e^{2m\bar{w}}dx& \le \frac{1}{C_1} \int_{\rn\setminus B_R}|x|^a|K(x)|e^{2m\bar{w}}dx\\
&\le- \int_{\rn\setminus B_R}x\cdot \nabla |K(x)|e^{2m\bar{w}}dx\\
& = 2m\int_{\rn}|K|e^{2m\bar{w}}dx +
 \int_{B_R}(x\cdotp\nabla |K(x)|)e^{2m\bar{w}}dx  \\
 & \quad -2mt|\alpha|\int_{B_1}(x\cdotp\nabla  v(x))(-\Delta)^mu_0dx\\
&=:(I) + (II) +(III),
\end{split}
\end{equation}
where in the equality on the third line we used \eqref{eqxnabla}.
Now using \eqref{defcv} and \eqref{expression_barw}, we compute $(I)=2mt|\alpha|\gamma_m$,
and using  Lemma \ref{local-derivatives} we bound
\begin{equation*}
\begin{split}
 (I)+(II)+(III)& \leq C_1 + \int_{B_ R}(x\cdotp\nabla |K(x)|)e^{2m\bar{w}}dx\\
 & \leq C_1+\int_{\Omega}(x\cdotp\nabla |K(x)|)e^{2m\bar{w}}dx
\end{split}
\end{equation*}
where
$$\Omega :=\left\{x\in\rn : x\cdot \nabla P(x)+\alpha<0\right\}.$$
From \eqref{stimapo} we infer that $\Omega\subset B_{R_1}$.
Then with Lemma \ref{local} we find
$$(I)+(II)+(III)\le C_1+\sup_{x\in B_{R_1}}(|x\cdot\nabla K(x)|)\int_{B_{R_1}}e^{2m\bar{w}}dx\le C_2=C_2(R_1), $$
where $C_2$ does not depend on $t$ or $v$.  To complete the proof
it suffices to take $R_0=R$ so large that
$$\frac{R^a}{C_1}\ge \frac{C_2}{\varepsilon}.$$
\end{proof}
 
 To prove uniform upper bound of $\bar{w}$ on the complement of a compact set, we use the Kelvin transform. For $R>1$ define
   \begin{align}\label{defn-xi}
\xi_R(x):=\bar{w}\left(\frac{Rx}{|x|^2}\right),\quad  0<|x|\leq 1.   
  \end{align}

 \begin{lem}\label{complement}
There exists $\varepsilon>0$ sufficiently small such that if $R_0=R_0(\varepsilon)>1$ is as in Lemma \ref{int<epsilon}, then $\xi(x):=\xi_{R_0}(x)$ is uniformly upper bounded on $B_1$, i.e. $\bar w$ is uniformly upper bounded in $\rn\setminus B_{R_0}$.
 \end{lem}
\begin{proof}
 Using (\ref{induction}) for $n=2m$ and $k=m$ and recalling that
$$(-\Delta)^m\bar w=Ke^{2m\bar w}\quad \text{in }\rn\setminus B_1, $$ 
we have
 \begin{align}
  (-\Delta)^m\xi(x)&=\frac{R_0^{2m}}{|x|^{4m}}((-\Delta)^m\bar{w})\left(\frac{R_0x}{|x|^2}\right) \notag\\
&  =\left(\frac{R_0}{|x|^2}\right)^{2m}K\left(\frac{R_0x}{|x|^2}\right)e^{2m\xi(x)}\notag\\
& =:f(x). \notag
 \end{align}
Then with the change of variable $y=\frac{R_0x}{|x|^2}$ and Lemma \ref{int<epsilon} we obtain for $R_0=R_0(\varepsilon)$ large enough (and $\varepsilon>0$ to be fixed later)
$$\int_{B_1}f(x) dx<\varepsilon.$$
We write $\bar{\xi}:=\xi_1+\xi_2$, where
\[
\left\{
\begin{array}{ll}
(-\Delta)^m\xi_1  =f & \text{in }B_1 \\
(-\Delta)^k\xi_1  =0 & \text{on }\partial B_1 \text{ for }k=0,1,2,..,m-1\\
\end{array}
\right.
\]
and
\[
\left\{
\begin{array}{ll}
(-\Delta)^m\xi_2  =0 & \text{in }B_1\\
(-\Delta)^k\xi_2  =(-\Delta)^k\xi& \text{on } \partial B_1 \text{ for }k=1,2,..,m-1 \\
\xi_2= \xi^+:=\max\{\xi,0\}& \text{on }\partial B_1. 
\end{array}
\right.
\]
Iteratively using the \emph{maximum principle} it is easy to see that 
\begin{equation}\label{xibarxi}
\xi\leq\bar{\xi}\text{ in }B_1.
\end{equation}

Now fix $\varepsilon>0$ small enough (and consequently $R_0=R_0(\varepsilon)>0$ large enough) so that by Lemma \ref{a2m} below, there exists $p>1$ 
such that $e^{2m\xi_1}$ is bounded in ${L^p(B_1)}$. As usual this bound, as well as $\varepsilon$, $R_0$ are independent of $t$ and $v$.  

Since $|\Delta^k\xi_2|$ is uniformly bounded on $\partial B_1$ for $k=0,1,2,...,m-1$ by Lemma \ref{local-derivatives} and 
$\bar{w}^+$ is uniformly  bounded on  $\partial B_{R_0}$ by Lemma \ref{local}, so that $\xi^+$ is uniformly bounded on $\partial B_1$,  by the maximum principle we get uniform bounds of $\xi_2$ in $B_1$.
 Hence, noticing that
$$\frac{R_0^{2m}}{|x|^{4m}}K\left( \frac{R_0x}{|x|^2}\right)\le C\quad \text{for }x\in B_1$$
by \eqref{decayK},  and using \eqref{xibarxi}, we can bound
 \begin{align}
  \|f\|_{L^p(B_1)} & \leq C \|e^{2m\xi}\|_{L^p(B_1)} 
 \notag\\
& \leq C \|e^{2m\bar \xi}\|_{L^p(B_1)}
 \notag\\
  &\leq C\|e^{2m\xi_1}\|_{L^p(B_1)} \| e^{2m\xi_2}\|_{L^\infty(B_1)}\notag\\ 
  &\leq C. \notag
 \end{align}
Consequently by elliptic estimates and Sobolev embedding there exists a conastant $C>0$ (independent of $v$ and $t$) such that 
$$\|\xi_1\|_{L^\infty(B_1)}\le C'\|\xi_1\|_{W^{2m,p}(B_1)}\leq C,$$
and therefore
$$\xi \le \bar \xi\le |\xi_1|+|\xi_2|\leq C \quad \text{in }B_1,$$
with $C$ not depending on $v$ and $t$.
\end{proof}

\section{Local uniform upper bounds for the equation $(-\Delta)^m u =Ke^{2mu}$}

Here we state a slightly simplified version of Theorem 1 from \cite{LM-quantization} which we will use to prove the uniform upper bound of Theorem \ref{lemma_1} below. This theorem was originally proved by F. Robert \cite{rob} in dimension $4$ and under the assumption $V_k>0$, and is a delicate counterpart to the blow-up analysis initiated by H. Br\'ezis and F. Merle \cite{BM} in dimension $2$. The crucial fact which we shall use is that in order to lose compactness $V_0$ must be positive somewhere \emph{and} $\|V_ke^{2mu_k}\|_{L^1}$ must approach or go above $\Lambda_1:=(2m-1)!\vol(S^{2m})$.

\begin{thm}[\cite{LM-quantization}]\label{useful}
 Let $\Omega\subseteq\rn$ be a connected set. Let $(u_k)\subset C^{2m}_{\mathrm{loc}}(\Omega)$ be such that
 \begin{align}
  (-\Delta)^mu_k=V_ke^{2mu_k}~~in~~\Omega \notag
 \end{align}
where $V_k\rightarrow V_0$ in $C^0_{\mathrm{loc}}(\Omega)$
and, for some $C_1,C_2>0$, 
\[
 \int_{\Omega}e^{2mu_k}dx\leq C_1, \quad  \int_\Omega |\Delta u_k|dx\leq C_2.
\]
Then one of the following is true:
\begin{itemize}
 \item [(i)] up to a subsequence $u_k\rightarrow u_0$ in $C^{2m-1}_{\mathrm{loc}}(\Omega)$ for some 
 $u_0\in C^{2m}(\Omega)$, or
 \item [(ii)] there is a finite (possibly empty) set $S=\{x^{(1)},....,x^{(I)}\}\subset\Omega$ such that $V_0(x^{(i)})>0$ for $1\le i\le I$, and up to a subsequence $u_k\rightarrow -\infty$ locally uniformly in $\Omega\setminus S$, and
\end{itemize}
\begin{align}
 V_ke^{2mu_k}dx\rightharpoonup \sum_{i=1}^I\alpha_i\delta_{x^{(i)}}  \notag
\end{align}
in the sense of measures in $\Omega$, where
$$\alpha_i=L_i\Lambda_1 \text{ for some } L_i\in\mathbb{N}\setminus\{0\},\quad \Lambda_1:=(2m-1)!\vol(S^{2m}).$$
In particular, in case $(ii)$ for any open set $\Omega_0\Subset\Omega$ with $S\subset\Omega_0$ we have
\begin{align}
 \int_{\Omega_0}V_ke^{2mu_k}\rightarrow L\Lambda_1 \text{ for some }L\in\mathbb{N},\text{ and }  L=0 \Leftrightarrow S=\emptyset.     \label{quant}
\end{align}
\end{thm}

\begin{thm}\label{lemma_1}
Let $u\in C^{2m}(B_R)$ solve
$$(-\Delta)^mu=Ke^{2mu} \quad\text{in } B_R$$
for a function $K\in C^0(B_R)$ and assume that for given $C_1,C_2>0$ one has
\begin{itemize}
  \item [($a$)] $\int_{B_R}e^{2mu} dx\leq C_1$,
  \item [($b$)] $\int_{B_R}|\Delta u|dx \leq C_2$, 
  \item [($c_1$)] either $\int_{B_R}Ke^{2mu} dx \le\Lambda$ for some $\Lambda < (2m-1)!|S^{2m}|$, or
\item [($c_2$)] $K\le 0$ in $B_R$.
 \end{itemize}
 Then
\begin{align}
 \sup_{B_{R/2}}u\leq C      \notag
\end{align}
where $C$ only depends on $R$, $C_1$, $C_2$, $\Lambda$ (in case ($c_1$) holds and not ($c_2$)) and $K$.
\end{thm}

\begin{proof}
 Assume that there is a sequence of functions $u_n\in C^{2m}(B_R)$ and a 
 sequence of points $x_n\in B_{R/2}$ such
 that $u_n$ satisfies the conditions ($a$), ($b$), and ($c_1$) or ($c_2$), and assume that
\begin{equation}\label{uninfty}
\lim_{n\to\infty}u_n(x_n)=\infty.
\end{equation}
 
Then we can apply Theorem \ref{useful} with $V_k=K$ for every $k$, and because of \eqref{uninfty}, we clearly are in case (ii) of the theorem.
Assume that $S\ne \emptyset$. Then $K>0$ on $S$, hence condition ($c_2$) does not hold. On the other hand condition ($c_1$) contradicts \eqref{quant}. Then $S=\emptyset$, hence $u_k\to -\infty$ uniformly in $B_{R/2}$, contradicting \eqref{uninfty}.
\end{proof}

\appendix

\section{Appendix}

\subsection{Some useful lemmas}

\begin{lem}[Pohozaev-type identity]\label{pohozaev}
Consider $K\in C^1(\overline{B_R})$ for some $R>1$, and let $u_0\in C^{2m}(\rn)$ be such that $\supp(\Delta^m u_0)\subseteq \overline{B_1}$. 
Let  $\bar{w}\in C^{2m}(\overline{B_R})$ be a  solution of
\begin{align}
 (-\Delta)^m\bar{w}=Ke^{2m\bar{w}}+t\alpha(-\Delta)^mu_0.      \notag
\end{align}
 Then we have
\begin{equation}\label{eqpoho1}
\begin{split} 
\int_{B_R}(x\cdotp\nabla K)e^{2m\bar{w}}dx &+ 2m\int_{B_R}Ke^{2m\bar{w}}dx -
  2mt\alpha\int_{B_1}(x\cdotp\nabla  \bar{w})(-\Delta)^mu_0dx \\
 &=R\int_{\partial B_R}Ke^{2m\bar{w}}d\sigma -mR\int_{\partial B_R}|\Delta^{\frac{m}{2}}\bar{w}|^2d\sigma -
 2m\int_{\partial B_R}fd\sigma,
\end{split}
\end{equation}
where,
$$\displaystyle f(x):=\sum_{j=0}^{m-1}(-1)^{m+j}\frac{x}{R}\cdotp\left(\Delta^{j/2}
 (x\cdotp\nabla \bar{w})\Delta^{(2m-1-j)/2}\bar{w}\right)\quad \text{on }\partial B_R,$$
and for $k$ odd $\Delta^{k/2}:=\nabla\Delta^{(k-1)/2}$.
\end{lem}

\begin{proof}
Integrating by parts we find
 \begin{align}
  2m\int_{B_R}(1+x\cdotp\nabla\bar{w})Ke^{2m\bar{w}}dx &=\int_{B_R}K\diver (xe^{2m\bar{w}}) dx \notag \\
 &=-\int_{B_R}(x\cdotp\nabla K)e^{2m\bar{w}} dx+R\int_{\partial B_R}Ke^{2m\bar{w}}  d\sigma.        \notag
 \end{align}
Now
\begin{equation}\label{eqpoho0}
 \int_{B_R}(x\cdotp\nabla\bar{w})Ke^{2m\bar{w}} dx=\int_{B_R}(x\cdotp\nabla\bar{w})(-\Delta)^m\bar{w} dx
 - t\alpha\int_{B_1}(x\cdotp\nabla\bar{w})(-\Delta)^mu_0  dx,  
\end{equation}
and integrating by parts $m$ times the first term on the right-hand side of \eqref{eqpoho0} we find
\begin{equation}\label{eqpoho}
 \int_{B_R}(x\cdotp\nabla\bar{w})(-\Delta)^m\bar{w}  dx
  = \int_{B_R}\Delta^\frac m2(x\cdotp\nabla\bar{w})\Delta^{\frac{m}{2}}\bar{w} dx+\int_{\partial B_R}f d\sigma=: I
\end{equation}
Using 
$$\Delta^\frac m2(x\cdotp\nabla\bar{w})\Delta^\frac m2\bar{w}=\frac{1}{2}\diver (x|\Delta^{\frac{m}{2}}\bar w|^2) $$
(see e.g. \cite[Lemma 14]{LM-asymptotic} for the simple proof) and using the divergence theorem we obtain
$$I= \frac{1}{2}\int_{\partial B_R}R|\Delta^{\frac{m}{2}}\bar{w}|^2 d\sigma + \int_{\partial B_R}f d\sigma,$$
and putting together the above equations we conclude.
\end{proof}

The proof of the following lemma can be found in \cite{LM-classification} (Theorem 7). It extends to arbitrary dimension Theorem 1 of \cite{BM}.

\begin{lem}\label{a2m} Let $f\in L^1(B_R)$ and let $v$ solve
$$\left\{
\begin{array}{ll}
(-\Delta)^m v=f &\textrm{in } B_R\subset \rn,\\
\Delta^k v=0& \textrm{on }\partial B_R \text{ for } k=0,1,\dots,m-1.
\end{array}
\right.
$$
Then, for any $p\in\Big(0,\frac{\gamma_m}{\|f\|_{L^1(B_R)}}\Big)$, we have $e^{2m p|v|}\in L^1(B_R)$ and
$$\int_{B_R}e^{2m p |v|}dx\leq C(p)R^{2m},$$
where $\gamma_m$ is definde by \eqref{gammam}.
\end{lem}

\begin{lem}
 Given $u\in C^\infty(\mathbb{R}^n)$, define $\tilde{u}(x):=u\big(\frac{x}{|x|^2}\big)$ for  
 $x\in \mathbb{R}^n\setminus \{0\}$. Then for any $k\in\mathbb{N}$ we have
 \begin{align}
  \Delta^k\left(\frac{1}{|x|^{n-2k}}\tilde{u}(x)\right)=\frac{1}{|x|^{n+2k}}
  (\Delta^ku)\left(\frac{x}{|x|^2}\right),\quad x\in \mathbb{R}^n\setminus\{0\}.\label{induction}   
 \end{align}
\end{lem}
\begin{proof}
We shall prove the lemma by induction on $k\in\mathbb{N}$. 
Notice that for $k=0$ \eqref{induction} is trivial.

For a smooth function $f$ and $g(x):=|x|^2$, we have the formula
  $$ \Delta^{k+1}(fg)=g\Delta^{k+1}f+2(k+1)(n+2k)\Delta^kf+4(k+1)x\cdotp\nabla(\Delta^kf),$$
which can be easily proven by induction on $k\in \mathbb{N}$.
Choosing
$$f(x)=\frac{\tilde{u}(x)}{|x|^{n-2k}}$$
and assuming that \eqref{induction} is
true for a given $k\in \mathbb{N}$, we compute
\begin{align}
\Delta^{k+1}\left(\frac{\tilde u(x)}{|x|^{n-2(k+1)}}\right)   =\;&\Delta^{k+1}(fg) \notag\\
=\;&g\Delta(\Delta^k f)+2(k+1)(n+2k)\Delta^kf+4(k+1)x\cdotp\nabla(\Delta^kf)\notag\\
 =\; &|x|^2\Delta\left
                  (\frac{1}{|x|^{n+2k}}(\Delta^ku)\left(\frac{x}{|x|^2}\right)\right)+2(k+1)(n+2k)\frac{1}{|x|^{n+2k}}(\Delta^ku)\left(\frac{x}{|x|^2}\right)\notag \\
    & +4(k+1)x\cdotp\nabla\left
                  (\frac{1}{|x|^{n+2k}}(\Delta^ku)\left(\frac{x}{|x|^2}\right)\right)\notag \\       
 =\; & \frac{1}{|x|^{n+2(k+1)}}(\Delta^{k+1}u)\left(\frac{x}{|x|^2}\right),        \notag         
\end{align}
hence completing the induction.
\end{proof}

\subsection{Proof of Lemma \ref{compact}}\label{proofcompact}

For any  $R\ge 1$ set
$$A_R:=\{x\in\rn: R<|x|<2R\},\quad A:=A_1=\{x\in\rn: 1<|x|<2\}.$$
Given $f\in W^{2m,p}(A_R)$, define
 \begin{align}
  \tilde{f}(x):=f(Rx), ~~\text{for $x\in A$}. \notag
 \end{align}
 For $|\beta|\leq 2m$, we have
\begin{align}
 \int_A|D^{\beta}\tilde{f}(x)|^pdx &= R^{p|\beta|}\int_{A}|(D^{\beta}f)(Rx)|^pdx \notag\\
 & = R^{p|\beta|-2m}\int_{A_R}|D^{\beta}f(x)|^pdx. \notag
\end{align}
From the embedding $W^{2m,p}(A)\hookrightarrow C^0(A)$ there exists a constant $S>0$,
such that 
\begin{align}
 \|u\|_{C^0(A)}\leq S\|u\|_{W^{2m,p}(A)}, ~~\text{for all $u\in W^{2m,p}(A)$}. \notag
\end{align}
Hence
\begin{equation}\label{vanishing_at_infty}
\begin{split}
 \|f\|_{C^0(A_R)} & =\|\tilde f\|_{C^0(A)} \\
 & \leq S\|\tilde{f}\|_{W^{2m,p}(A)} \\
 & =S\sum_{|\beta|\leq 2m}\|D^{\beta}\tilde{f}\|_{L^p(A)} \\
 & = S\sum_{|\beta|\leq 2m}R^{|\beta|-2m/p}\|D^{\beta}f\|_{L^p(A_R)} \\
 & \le C S\sum_{|\beta|\leq 2m}R^{-2m/p-\delta}
          \|(1+|x|^2)^{\frac{\delta+|\beta|}{2}}D^{\beta}f\|_{L^p(A_R)} \\
 & \leq CSR^{-\gamma}\|f\|_{M_{2m,\delta}^p}, \quad \gamma=2m/p+\delta>0.        
\end{split}
\end{equation}
Since $R\ge 1$ is arbitrary \eqref{vanishing_at_infty} and on $B_2$ we have
\begin{equation}\label{sob}
\|f\|_{C^0(B_2)}\le S'\|f\|_{W^{2m,p}(B_2)}\le CS' \|f\|_{M_{2m,\delta}^p},
\end{equation}
we conclude that $M_{2m,\delta}^p(\rn)\subset C_0(\rn)$, and actually 
\begin{equation}\label{limfnR}
\sup_{n\in\mathbb{N}} \|f_n\|_{M_{2m,\delta}^p}<\infty\quad \Rightarrow\quad \lim_{R\to\infty} \sup_{n\in \mathbb{N}}\|f_n\|_{C^0(A_R)}= 0.
\end{equation}
By \eqref{vanishing_at_infty} and \eqref{sob}, on any compact set $\Omega\Subset\rn$ the sequence $\|f_n\|_{W^{2m,p}(\Omega)}$ is bounded and from the compact embedding $W^{2m,p}(\Omega)\hookrightarrow C^0(\Omega)$, we can extract a subsequence converging in $C^0(\Omega)$. Then up to choosing $\Omega=B_n$ and extracting a diagonal subsequence we have $f_n\to f$ locally uniformly for a continuous function $f$, and actually $f\in C_0(\rn)$ and the convergence is globally uniform thanks to \eqref{limfnR}.
\hfill $\square$

\subsection{Condition \eqref{poly3} does not imply \eqref{poly}}

\begin{prop}\label{proppoly} For $n\ge 2$ there exists a polynomial $P$ of degree $4$ in $\mathbb{R}^n$ satisfying \eqref{poly3} but not \eqref{poly}.
\end{prop}
 
\begin{proof} In $\mathbb{R}^2$ consider $P(x)=P(x_1,x_2)=x_1^2+ x_2^4-\beta x_1x_2^2$, with $\beta<2$.
Then
$$P(x)\ge x_1^2+x_2^4-\beta\left( \frac{x_1^2}{2}+\frac{x_2^4}{2}\right) =\left(1-\frac \beta 2\right) (x_1^2+x_2^4),$$
so that $P$ satisfies \eqref{poly3}. Moreover
$$x \cdot \nabla P(x)=2x_1^2+4 x_2^4-3\beta x_1x_2^2.$$
Choosing $x=(ax_2^2,x_2)$ we obtain
$$(ax_2^2,x_2)\cdot \nabla P(ax_2^2,x_2)=x_2^4(2a^2-3\beta a+4).$$
Then, since for $|\beta|>\sqrt{2}\frac{4}{3}$ the polynomial $2a^2-3\beta a+4$ has positive discriminant, fixing $\beta \in (-\infty,-\sqrt{2}\frac{4}{3})\cup  \left(\sqrt{2}\frac{4}{3},2\right)$  and $a$ such that $2a^2-3\beta a+4<0$ we see that 
$$\liminf_{|x|\to\infty} x\cdot \nabla P(x)\le \lim_{|x_2|\to\infty }(ax_2^2,x_2)\cdot \nabla P(ax_2^2,x_2)=-\infty.$$
This proves the proposition for $n=2$. For $n>2$ it suffices to consider
$$\tilde P(x_1,x_2,\dots,x_n)=P(x_1,x_2)+ \sum_{j=3}^n x_j^2,$$
where $P$ is as before. 
\end{proof}

 \end{document}